\documentclass[11pt]{amsart}
\usepackage{amsmath, amsthm, amssymb}
\usepackage{graphicx, array}
\usepackage{longtable}
\usepackage{epstopdf}

\newtheorem{theorem}{Theorem}[section]
\newtheorem{Prop}{Proposition}
\newtheorem{lemma}{Lemma}
\newtheorem{Cor}{Corollary}

\theoremstyle{definition}
\newtheorem{definition}{Definiton}
\newtheorem{remark}{Remark}

\title{A structural property of Adian inverse semigroups}

\bigskip

\author{Muhammad Inam}
\email{s-minam1@math.unl.edu}

\author{John Meakin}
\email{jmeakin@math.unl.edu}

\author{Robert Ruyle }
\email{rruyle3@unl.edu}
\address{Department of Mathematics\\
University of Nebraska-Lincoln\\
Lincoln, NE 68588, USA}

\begin{document}

\date{\today}

\begin{abstract}

We prove that an inverse semigroup over an Adian presentation is
$E$-unitary.

\end{abstract}

\maketitle

\begin{center}
Dedicated to the memory of John M. Howie
\end{center}

%%We  also prove that if $(X,R)$ is an Adian
%%presentation where $R=\{(u_i,v_i) : i\in I\}$, then the inverse
%%monoid given by the presentation $Inv\langle X : u_iv_i^{-1}=1, \, i
%%\in I\rangle$ is $E$-unitary.

\section
{Introduction}

Throughout this paper, $X$ is a non-empty set and
$R=\{(u_i,v_i):i\in I\}$ where
 $u_i,v_i\in X^+$. We refer to the pair $P=(X, R)$ as a {\it positive
 presentation}. The semigroup with the set $X$ of generators and the set $R$ of relations will be
 denoted by $Sgp\langle X \vert R\rangle$ and the corresponding group  will be denoted by
 $Gp\langle X \vert R\rangle$.  We  obtain two undirected
 graphs corresponding to $P$ as follows. The {\it left graph} $LG(P)$ of $P$ has set $X$ of
 vertices and
  has an undirected edge connecting $x$ and $y$ in $X$ corresponding to each relation of the
  form $(u_i,v_i)$ in $R$ where $x$ is the first letter of $u_i$ and $y$ is the first
  letter of $v_i$. The {\it right graph} $RG(P)$ is constructed dually, with edges connecting the
  last letters of pairs $(u_i,v_i)$ in $R$.  If there is no cycle in either the left graph or
  the right graph of a
  positive presentation, then
 the presentation is called a {\it cycle-free presentation} or an {\it Adian presentation}. These
 presentations were first studied by S. I. Adian in \cite{AD}, where it was shown
 that the semigroup $Sgp\langle X \vert R\rangle$ embeds in the group $Gp\langle X \vert R\rangle$.
 Unless stated otherwise,  we will consider our presentations to be Adian presentations throughout.

 A semigroup $S$ is called an {\it inverse semigroup} if for each $a\in S$ there exists a
 unique element $b\in S$ such that $aba=a$ and $bab=b$. This unique element $b$ is denoted
 by $a^{-1}$. It is well known that  idempotents commute in an inverse
 semigroup and that the product of two idempotents is an idempotent.
 If $S$ is an inverse semigroup, then the natural partial order on $S$ is defined, for
 $a, b \in S$, by
 $a\leq b$ if and only if $a=aa^{-1}b$.
% If $x,y \in S$ and $e$ is an idempotent of $S$
% then it is standard to see that
% $xey \leq xy$ and also that $xex^{-1}$ is an idempotent of $S$. Also, if $e$ is an idempotent
% and $a \leq e$ in $S$, then $a$ is an idempotent.
 A congruence relation $\sigma$ is defined on $S$, for $a, b \in S$, by $a\sigma b$ if
 and only if there exists an element
 $c\in S$ such that $c\leq a,b$. Then $\sigma$ is the minimum group congruence
 relation on $S$, i.e., $S/\sigma $ is the maximum group homomorphic image of
 $S$. We refer the reader to the text \cite{IS} for these facts and further details about
 inverse
 semigroups. The following proposition concerning idempotents and the natural partial order
 is standard in the literature on inverse semigroups; we will use these facts often in this paper.

 \begin{Prop}\label{prop0}
 Let $S$ be an inverse semigroup.
 \begin{enumerate}
 \item[(i)] If $e$ is an idempotent and $a \leq e$ in $S$, then $a$ is an idempotent.
 \item[(ii)] If $x \in S$ and $e$ is an idempotent of $S$, then $xex^{-1}$ is an idempotent.
 \item[(iii)] Suppose $x_0, \dots, x_n \in S$ and suppose $e_1, \dots, e_n$ are idempotents.
 Then $x_0e_1x_1 \cdots e_nx_n \leq x_0x_1\cdots x_n$.
 In particular, if $x_0x_1 \cdots x_n$ is an idempotent, then $x_0e_1x_1 \cdots e_nx_n$ is an
 idempotent.
 \end{enumerate}
 \end{Prop}

 Inverse semigroups form a variety of algebras in the sense of universal algebra,
 and so free inverse semigroups exist. The inverse semigroup with the set $X$ of generators and
 the set $R$ of relations is denoted by $Inv\langle
 X \vert R\rangle$. This is the quotient of the free inverse semigroup on
 $X$ obtained by imposing the relations $u_i = v_i$ in $R$.
  See Stephen's paper \cite{SG} for basic information
 about presentations of inverse semigroups.
 A basic fact is that if an inverse semigroup $S$ is given by a
 presentation
 $S = Inv\langle X \vert R\rangle$ then the group $S/\sigma$ is given by the presentation
 $Gp\langle X \vert R\rangle$.

 An inverse semigroup $S$ is called {\it $E$-unitary} if the natural homomorphism
from $S$ onto $G = S/\sigma$ is
 {\it idempotent-pure}, i.e., if the inverse image of the identity of
 $G$ is equal to the set of idempotents of $S$. We refer to \cite{IS}
 for equivalent definitions and further information about the importance of $E$-unitary
 inverse semigroups in the
 literature on inverse semigroups.

 In this paper we prove that if $P=(X, R)$ is an Adian
 presentation, then the inverse semigroup  $S = Inv\langle
 X \vert R\rangle$ is $E$-unitary.

 \section
 {Main Theorem}

 In order to set the foundations of our main theorem, we first give a common definition
 of a
 {\it van Kampen diagram} over a presentation $(X, R)$ and then  describe some properties of
 van Kampen diagrams  over an Adian presentation $(X, R)$.

 A van Kampen diagram $\Delta$ is a planar, finite, connected, simply
 connected, oriented 2-complex (often called a ``map"). There is an involution
 $e \rightarrow e^{-1}$ on the edges of the diagram such that $(e^{-1})^{-1} = e$
 and $e^{-1} \neq e$. The pair $\{e,e^{-1}\}$ may be considered as an undirected edge
 with one of these edges viewed as positively oriented and the other as negatively oriented.
 A van Kampen diagram over the
 presentation $(X,R)$ satisfies:
 \begin{enumerate}
 \item[(i)] Each directed positively oriented edge of the 1-skeleton is
 labeled by a letter in $X$ and its inverse is labeled by the inverse of that letter.
 \item[(ii)] The 2-cells (sometimes called
 ``regions" or  ``faces") are homeomorphic to open disks. The boundary of each 2-cell  is
 labeled by a cyclically reduced conjugate of
 $uv^{-1}$ for some relation $(u,v) \in R$.
 \item[(iii)] There is a distinguished base vertex $O$ that lies on the topological boundary of
 $\Delta$.
 \end{enumerate}

The boundary cycle of $\Delta$ that begins and ends at the base
vertex $O$, and travels in the counterclockwise direction around the
boundary, is labeled by some  (not necessarily freely reduced) word
$w \in (X\cup X^{-1})^*$. In this case, we say that $\Delta$ is a
van Kampen diagram for $w$ over $(X,R)$.

A van Kampen diagram $\Delta$ is said to be {\it reduced} if it
satisfies the additional  condition:
\begin{enumerate}
\item[(iv)] There do not exist two 2-cells with an edge in common to their boundaries
that are mirror images of each
other.

\end{enumerate}

In this paper we will be working mainly with reduced van Kampen
diagrams. The well known Lemma of van Kampen says: (1) The boundary
label of a reduced van Kampen diagram over $(X,R)$ is a  word $w$
that is equal to the identity in $G = Gp\langle X \vert R\rangle$.
(2) For every every  word $w$ representing the identity of $G$ there
exists a reduced van Kampen diagram that has $w$ as its boundary
label.

We remark that some authors require a van Kampen diagram to have a
{\it reduced} word as its boundary label, and in fact some require
the boundary label to be a {\it cyclically reduced} word. It is
convenient for our purposes to make no such restriction however.
 We refer to the book by Lyndon and Schupp
 \cite{LS} for more detail and for basic information about van Kampen
 diagrams. Such diagrams are referred to as {\it $g$-diagrams} in the
 paper of Remmers \cite{RM}. A $g$-diagram in Remmers' sense is
 required to have a cyclically reduced word as boundary label, but
 this restriction is in fact not essential.

\begin{remark}\label{rem0}
%%A word $w \in (X \cup X^{-1})^*$ is equal to the identity
%%in the group $G = Gp\langle X : R \rangle$ if and only if its
%% reduced form $r(w)$ is equal to the identity in $G$. Note also that a
%%word $w \in (X \cup X^{-1})^*$ is an idempotent in the inverse
%%semigroup $S  = Inv\langle X \vert R\rangle$ if its reduced form is an
%%idempotent of $S$. This follows from Proposition \ref{prop0}, which implies in particular
%%that $w \leq r(w)$ in $S$.
It follows from van Kampen's Lemma that the inverse semigroup
 $S = Inv\langle X \vert R\rangle$ is $E$-unitary if and only if, for every van Kampen diagram
 $\Delta$ over $(X, R)$, the   word $w$ labeling the boundary cycle of $\Delta$ that
 starts and ends at the distinguished vertex
 $0$ is an idempotent of $S$.

%%We remark that some authors define van Kampen diagrams to be
%%diagrams of the above type whose boundary cycles are labeled by
%%reduced (not necessarily {\it cyclically} reduced) words that are
%%equal to the identity in $G$. Such diagrams differ from the diagrams
%%discussed above in that they may possibly have one extremal vertex
%%attached to the remainder of the diagram by an arc. It will be more
%%convenient for our purposes to restrict the definition as above so
%%that a boundary cycle of the diagram is labeled by  a cyclically
%%reduced word. However, later in the paper we will introduce the
%%notion of a {\it generalized van Kampen diagram} that is obtained
%%from a van Kampen diagram by attaching finitely many trees to the
%%diagram.
\end{remark}

%%\noindent {\bf Properties of van Kampen diagrams over an Adian presentation.}
 Each 2-cell of a van Kampen diagram over an Adian presentation $P = (X, R)$
 is {\it two-sided}. The boundary of a such a 2-cell
 can be viewed as a birooted graph, where one side of a relation $(u,v)\in R$ labels a path $p$ from
 the initial vertex to the terminal vertex of the graph, while the other side of the relation labels
 a different path $q$ from the initial vertex to the terminal vertex of the graph. We refer to these
 paths as the {\it sides} of the 2-cell.  The path $pq^{-1}$ is a boundary cycle of the 2-cell.
 %% Since the presentation
 %%is cycle free, no two boundary edges of a 2-cell fold together in the diagram.
 %%This implies that the closure of each 2-cell
 %%is homeomorphic to a closed disk.

 A vertex $\alpha$ of a van Kampen diagram $\Delta$ is called a {\it
 source} if all edges of $\Delta$ with initial vertex $\alpha$ are
 positively labeled. {\it Sinks} of $\Delta$ are defined dually. The
 following fact about van Kampen diagrams over Adian presentations
 was proved by Remmers \cite{RM} (Theorem 4.3) for diagrams whose
 boundary label is a cyclically reduced word: however, Remmers' proof
 carries over verbatim for the more general notion of van Kampen
 diagrams that we are considering in this paper.

\begin{lemma} \label{lem0}
Let $\Delta$ be a reduced van Kampen diagram for a word $w$ over an Adian
presentation $(X, R)$. Then $\Delta$ has no interior sources and no interior sinks.
\end{lemma}

%%Note that in Remmers \cite{RM} a diagram is defined to be reduced if it has no interior sources
%%and sinks, which is not the conventional definition of a reduced diagram in the literature. Since
%%we are only considering van Kampen diagrams over Adian
%%presentations this will not cause confusion, because Lemma \ref{lem0} guarantees that a diagram
%%over an Adian presentation that is reduced in the conventional sense is also reduced in the sense
%%of Remmers. For this paper we shall assume that diagrams are reduced in the conventional sense and
%%thus also reduced in the sense of Remmers.

The following lemma follows from the above discussion and from Lemma 2.2
and Theorem 4.3 of \cite{RM}.

\begin{lemma}  If $\Delta$ is a reduced van Kampen diagram over an Adian presentation,
then $\Delta$ satisfies the following conditions:
\begin{enumerate}
 \item[(i)] Every 2-cell of $\Delta$ is two-sided.

 \item[(ii)] $\Delta$ contains no directed (i.e., positively labeled)  cycles.

 \item[(iii)] Every positively labeled interior edge of $\Delta$ can be extended to a directed
   transversal of $\Delta$.
\end{enumerate}
\end{lemma}

Here, by a {\it directed transversal} of $\Delta$  we mean a
positively labeled path between two distinct boundary vertices of
$\Delta$, all of whose vertices are distinct,  all of whose edges
are interior to $\Delta$, and all of whose vertices except the
initial and terminal vertex are interior vertices in $\Delta$.

The following result is the main theorem of the paper.

\begin{theorem}\label{main}

An inverse semigroup $S=Inv\langle X \vert R\rangle$ over an Adian
presentation is E-unitary.

\end{theorem}
Before proving this theorem, we introduce some definitions and prove
some lemmas.

%%By the {\it closure} of a van Kampen diagram we mean the union of
%%all of its 0-cells, 1-cells and 2-cells in the Euclidean plane.

\begin{definition}
A subdiagram ${\Delta}'$ of a van Kampen diagram $\Delta$ is called
a {\it simple component} of $\Delta$ if it is a maximal subdiagram
whose boundary is a simple closed curve.
\end{definition}

We remark that our definition does not require that the boundary of
a simple component of reduced diagram $\Delta$ needs to be labeled
by a reduced word. However, since no $2$-cell of a van Kampen
diagram over an {\em Adian presentation} has a boundary that
contains an extremal vertex (i.e. a vertex of degree $1$), it
follows that a simple component of such a diagram does not contain
any attached trees as part of its boundary. A simple component of 
 $\Delta$ is itself a van Kampen diagram.
The diagram $\Delta$ has a tree-like (or ``cactoid") structure of
simple components connected by (possibly trivial) arcs and possibly
with finitely many finite trees attached to the boundary of the
diagram (see Figure \ref{fig0}). A van Kampen diagram with no simple
components is just a finite (edge-labeled) tree.

\begin{figure}[h!]
\centering
\includegraphics[trim = 0mm 0mm 0mm 0mm, clip,width=3in]{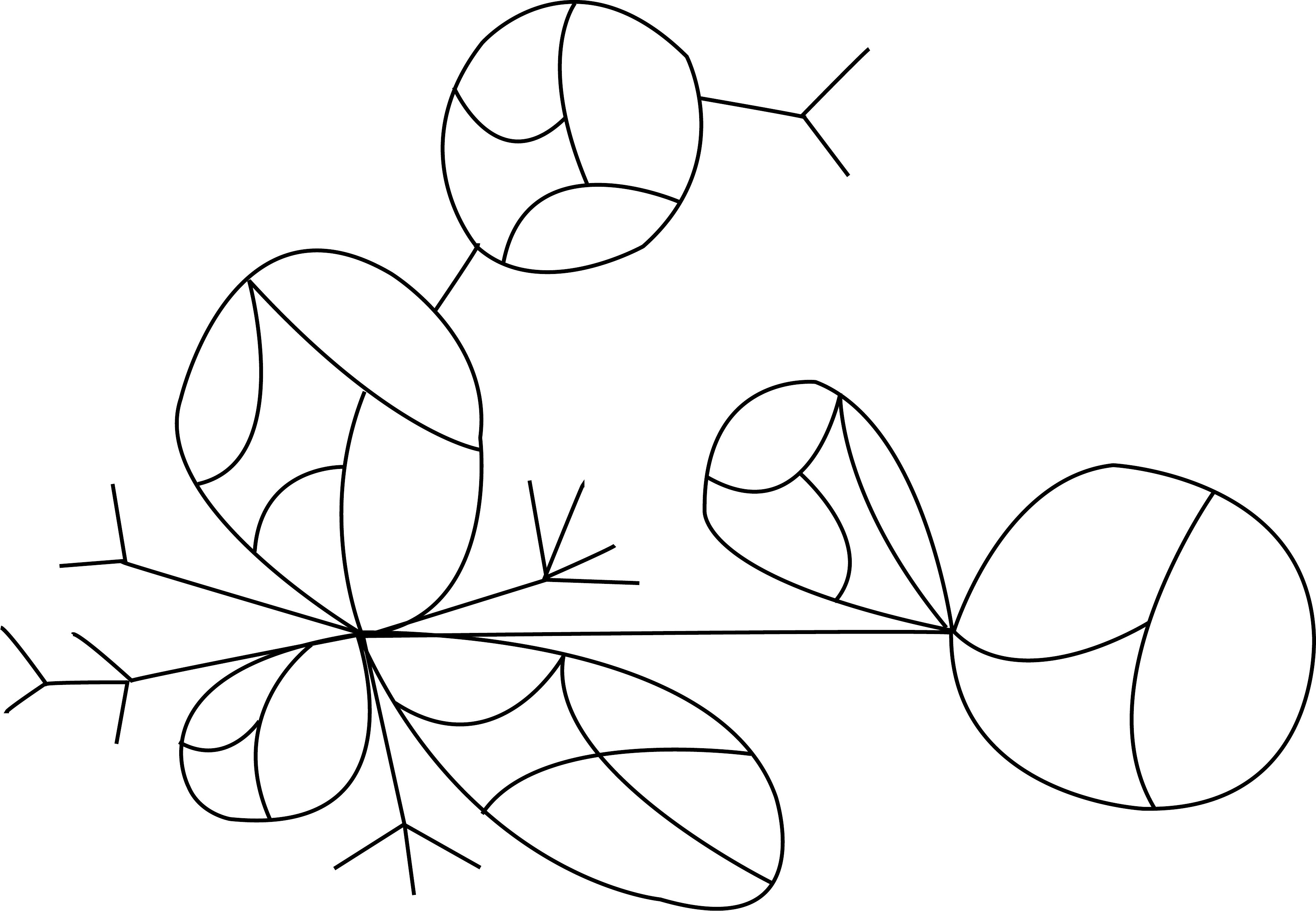}
\caption{ }
\label{fig0}
\end{figure}

%%In general, it is possible for a reduced word $w$ that a reduced diagram $\Delta$ for $w$
%%may have a simple component $D$ whose boundary label $z$ is not cyclically reduced. However
%%as remarked above, such a simple component $D$ is itself a (possibly non-reduced) van Kampen
%%diagram for the word $z$. Thus by van Kampen's Lemma, there exists a reduced diagram $D'$ that
%%has boundary label equal to a cyclically reduced conjugate of $z$. Thus, we may assume without
%%loss of generality that for any presentation $(X, R)$ and for any reduced word $w$, there exists
%%a reduced van Kampen diagram $\Delta$ for $w$ that has the further property that each of its
%%simple components has a boundary label that is cyclically reduced. In such a diagram $\Delta'$
%%we have that each simple component is in fact a reduced van Kampen diagram for the word that
%%labels its boundary, starting at any base vertex on its boundary.

\begin{definition}

For a van Kampen diagram $\Delta$ over an Adian presentation
$(X, R)$, a {\it transversal subdiagram} $\Delta'$ is a subdiagram of a simple component
of $\Delta$ such that $\Delta'$ has a boundary cycle of the form $pq$, where $p$ is a
directed transversal and $q$ is a subpath of a boundary cycle of the simple component in which $\Delta'$ is contained.

%labeled by a
%word $ts$, where $t\in X^+$ labels a
%directed transversal of the simple component and $s\in (X\cup X^{-1})^*$ labels a
%subpath of a boundary cycle of the simple component.

\end{definition}

\begin{lemma}\label{lem1}

If $\Delta$ is a van Kampen diagram with exactly one simple
component and no extremal vertex, then $\Delta$ has a directed
transversal if and only if it has more than one 2-cell. Furthermore,
any directed transversal of $\Delta$ divides $\Delta$ into two
transversal subdiagrams, each of which may be viewed as a van Kampen
diagram with exactly one simple component and no extremal vertex.

\end{lemma}

\begin{proof}
If $\Delta$ has a directed transversal, then it must have an interior
edge and hence it must have more than one 2-cell. Conversely if
$\Delta$ has more than one 2-cell and has just one simple component,
it must have an interior edge $e$. We extend this edge to a directed
 transversal $T$ of $\Delta$ by using Lemma 2 above. Denote
the initial vertex of $T$ by $\alpha$ and the terminal vertex of $T$
by $\beta$. Thus $\alpha$ and $\beta$ are on the boundary of
$\Delta$. This transversal $T$ divides $\Delta$ into two proper
subdiagrams ${\Delta}_1$ and ${\Delta}_2$ (see Figure \ref{fig1}).
Both ${\Delta}_1$ and ${\Delta}_2$ are clearly transversal
subdiagrams of $\Delta$. Also, the boundary of  $\Delta$ is  a
simple closed curve, and $T$ is a path between distinct boundary
vertices of $\Delta$ with all of its vertices distinct and all of
its vertices except $\alpha$ and $\beta$ interior to $\Delta$. It
follows that the boundary of each subdiagram ${\Delta}_i$ is
topologically a simple closed curve.  Hence each subdiagram
${\Delta}_i$ may be viewed as a van Kampen diagram with exactly one
simple component and no extremal vertex.
\end{proof}

\begin{figure}[h!]
\centering
\includegraphics[trim = 0mm 0mm 0mm 0mm, clip,width=2in]{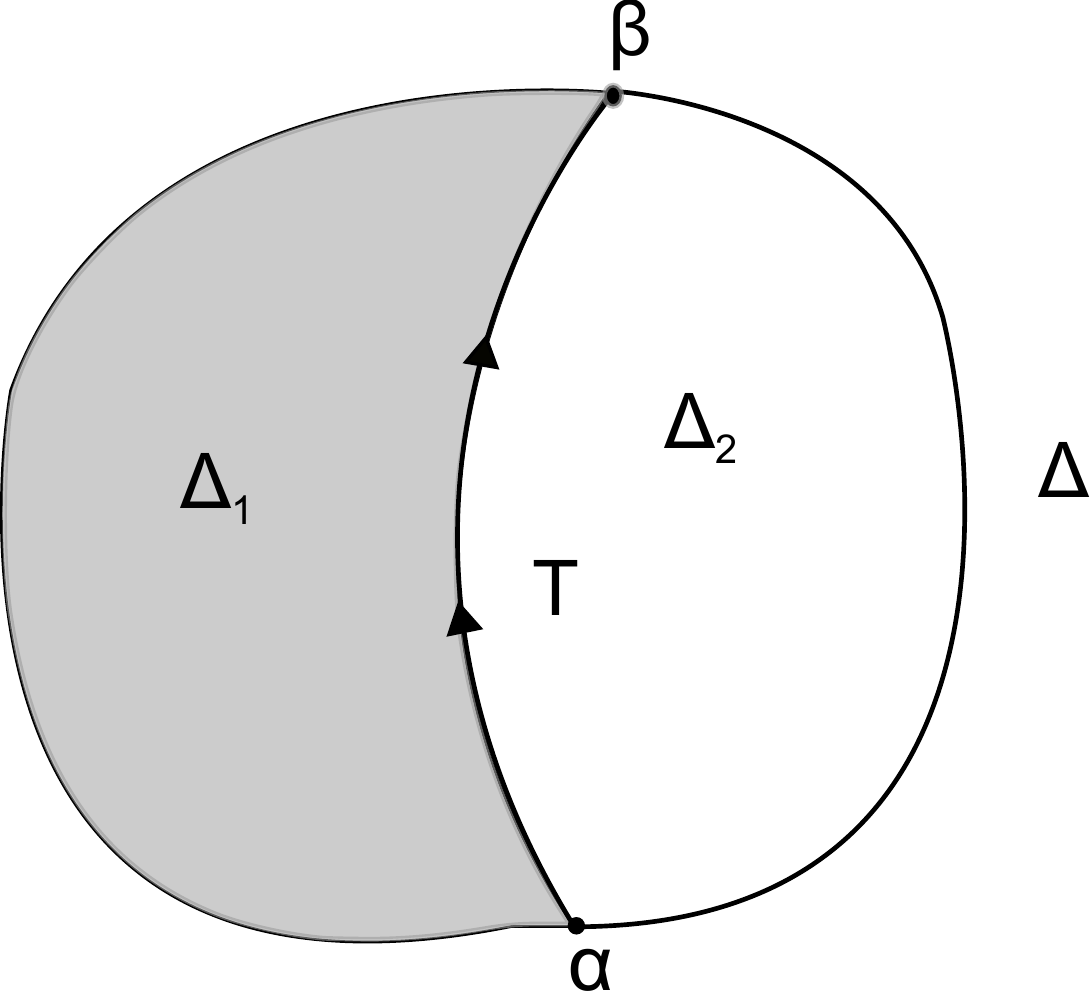}
\caption{A diagram $\Delta$ with two transversal subdiagrams.}
\label{fig1}
\end{figure}

\begin{definition}
For a van Kampen diagram $\Delta$ over an Adian presentation
$(X, R)$, a {\it special 2-cell} is a  2-cell, one
of whose two sides lies entirely on the boundary of $\Delta$.

\end{definition}

%%\begin{remark}

%%Note that, if a van Kampen diagram $\Delta$ over an Adian
%%presentation consists of only one 2-cell or if each of its simple
%%components consists of only one 2-cell, then $\Delta$ cannot have a
%%transversal subdiagram since it has no interior edges.

%%\end{remark}

\begin{lemma}\label{3}

%%Let $\Delta$ be a van Kampen diagram over an Adian presentation
%%$(X,R)$, such that at least one of its components contains more than
%%one 2-cell. Then $\Delta$ contains a minimal transversal subdiagram,
%%that consists of only one 2-cell.

Let $\Delta$ be a van Kampen diagram over an Adian presentation
$(X, R)$ that has more than one 2-cell and that has
just one simple component and no extremal vertex. Then $\Delta$
contains at least two special 2-cells.

\end{lemma}

\begin{proof}

%%Let $C$ be a simple component of a van Kampen diagram $\Delta$ over
%%an Adian presentation $(X,R)$ that  consists of more than one
%%2-cell.

%\newpage

%%To see this, let $u$ be the word that
%%labels the directed transversal $T$ of $\Delta$, and let $v'$ and
%%$v''$ be the words that label the paths along the boundary of
%%$\Delta$ from the terminal vertex of $T$ to the initial vertex of
%%$T$, around $\Delta_1$ and $\Delta_2$, respectively. Then the
%%boundary cycles labeled by $uv'$ and $uv''$ of $\Delta_1$ and
%%$\Delta_2$ respectively, satisfy the definition of a transversal
%%subdiagram.

As in the proof of Lemma \ref{lem1} we may choose a directed
transversal $T$ from a vertex $\alpha$ on $\partial \Delta$ to a
vertex $\beta$ on $\partial \Delta$ that divides the diagram
$\Delta$ into two subdiagrams $\Delta_1$ and $\Delta_2$.

   If both  of the subdiagrams  $\Delta_1$ and $\Delta_2$ consist of only one 2-cell,
   then we are done. Otherwise, we repeatedly subdivide each of these subdiagrams  to find
    special $2$-cells of $\Delta$. Without loss of generality, we pick $\Delta_1$
   and find
   an interior edge $e_1$ in $\Delta_1$. We extend $e_1$ to a directed transversal $T_1$
   of $\Delta_1$. Let ${\alpha}_1$ be the initial vertex of $T_1$ and ${\beta}_1$ the terminal
   vertex of $T_1$.
   This transversal $T_1$  divides $\Delta_1$ into two proper
   subdiagrams.
   %%$\Delta_{1_1}$
   %%and $\Delta_{1_2}$
   We claim that at least one of them is a transversal subdiagram of $\Delta$.
   To see this we
   consider the following four cases.

 \textit{Case 1:} Suppose ${\alpha}_1$ and ${\beta}_1$ lie on the boundary
 $\partial \Delta_1\cap\partial\Delta$. Since $T_1$ is a positively labeled path,
 the collection of 2-cells and their boundaries
 bounded by the path $T_1$ and the path from ${\beta}_1$ to  ${\alpha}_1$
  along $\partial \Delta_1\cap\partial\Delta$ forms a transversal subdiagram of $\Delta$,
 which contains fewer 2-cells than $\Delta_1$ (see Figure \ref{fig2}).

\begin{figure}[h!]
\centering
\includegraphics[trim = 0mm 0mm 0mm 0mm, clip,width=2in]{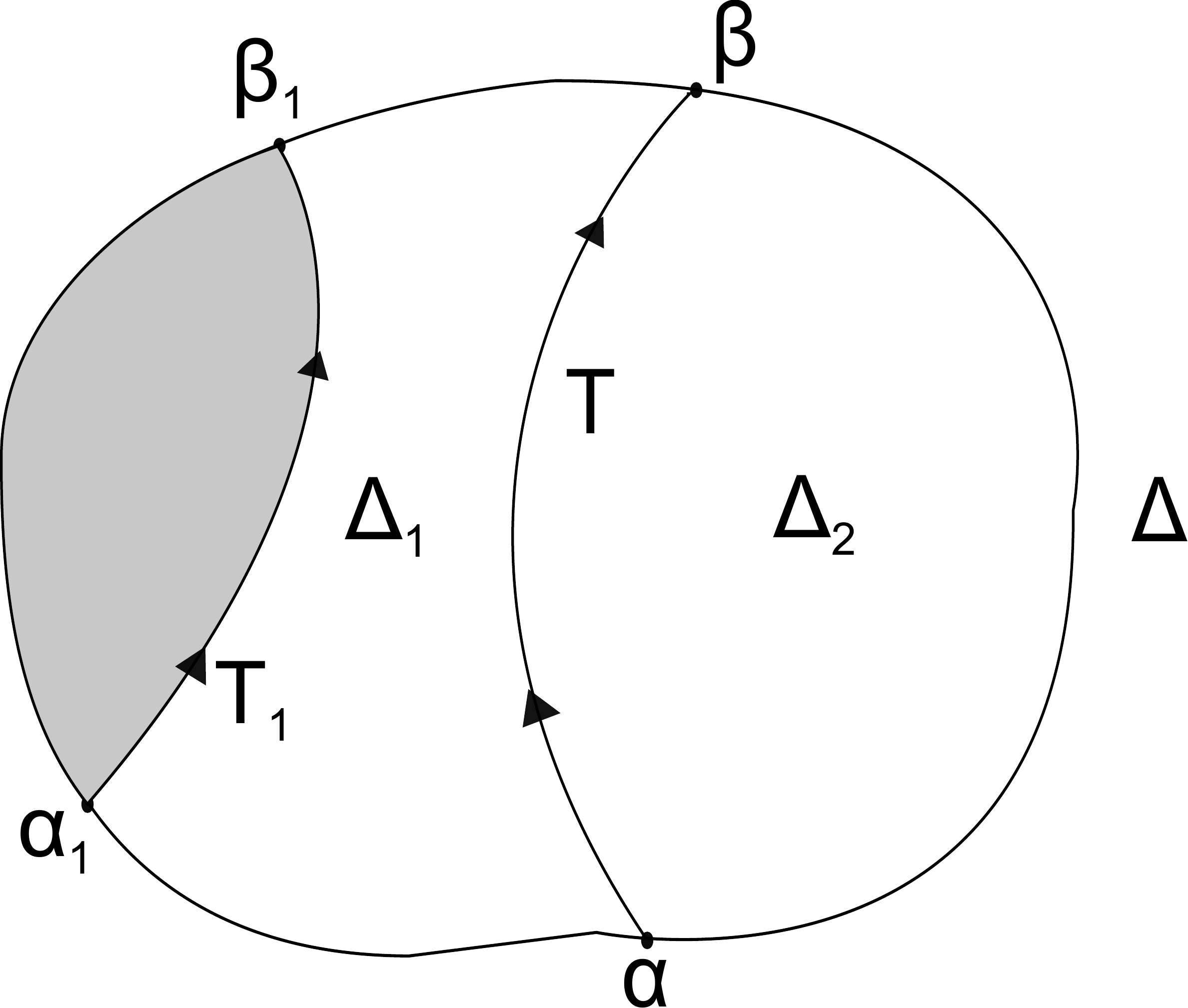}
\caption{ }
\label{fig2}
\end{figure}

\textit{Case 2:} Suppose both ${\alpha}_1$ and ${\beta}_1$ lie on
the transversal $T$. Note that if the transversals $T$ and $T_1$ are
oppositely oriented, then our van Kampen diagram $\Delta$ contains a
directed (positively labeled) cycle, which contradicts Lemma 2
above. Hence both transversals will have the same orientation. Let
$T_2$ be the composition of the path along $T$ from $\alpha$ to
${\alpha}_1$ followed by the path $T_1$, followed by the path along
$T$ from ${\beta}_1$ to $\beta$ (see Figure \ref{fig3}). Clearly
$T_2$ is a positively labeled path that is a transversal of
$\Delta$.

\begin{figure}[h!]
\centering
\includegraphics[trim = 0mm 0mm 0mm 0mm, clip,width=3.5in]{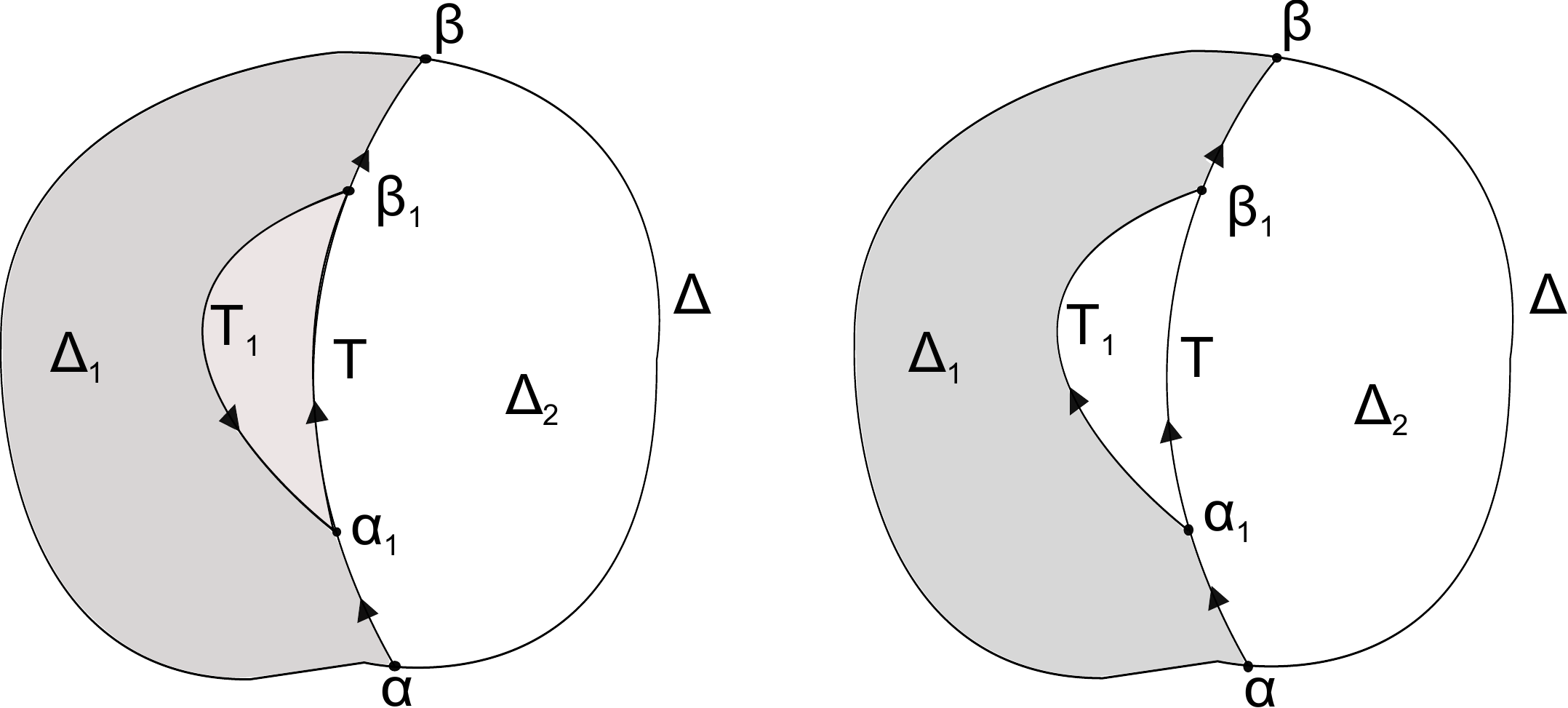}
\caption{ }
\label{fig3}
\end{figure}

The collection of 2-cells and their boundaries bounded by $T_2$ and
the path that goes from $\beta$ to $\alpha$ along $\partial
\Delta_1\cap\partial\Delta$ forms a transversal subdiagram of
$\Delta$ that contains fewer 2-cells than $\Delta_1$.

\textit{Case 3:} Suppose ${\alpha}_1$ lies on the boundary $\partial
\Delta_1\cap\partial\Delta$ and ${\beta}_1$ lies on the path $T$.
Let $T_2$ be the composition of the path $T_1$ and the path along
$T$ from ${\beta}_1$ to $\beta$ (see Figure \ref{fig4}). Again,
$T_2$ is a positively labeled path that is a transversal of
$\Delta$.
 The
 collection of 2-cells bounded by $T_2$ and the path from $\beta$ to
 ${\alpha}_1$ along $\partial\Delta_1\cap\partial\Delta$
 forms a transversal subdiagram of $\Delta$ that contains fewer regions than $\Delta_1$.

\begin{figure}[h!]
\centering
\includegraphics[trim = 0mm 0mm 0mm 0mm, clip,width=2in]{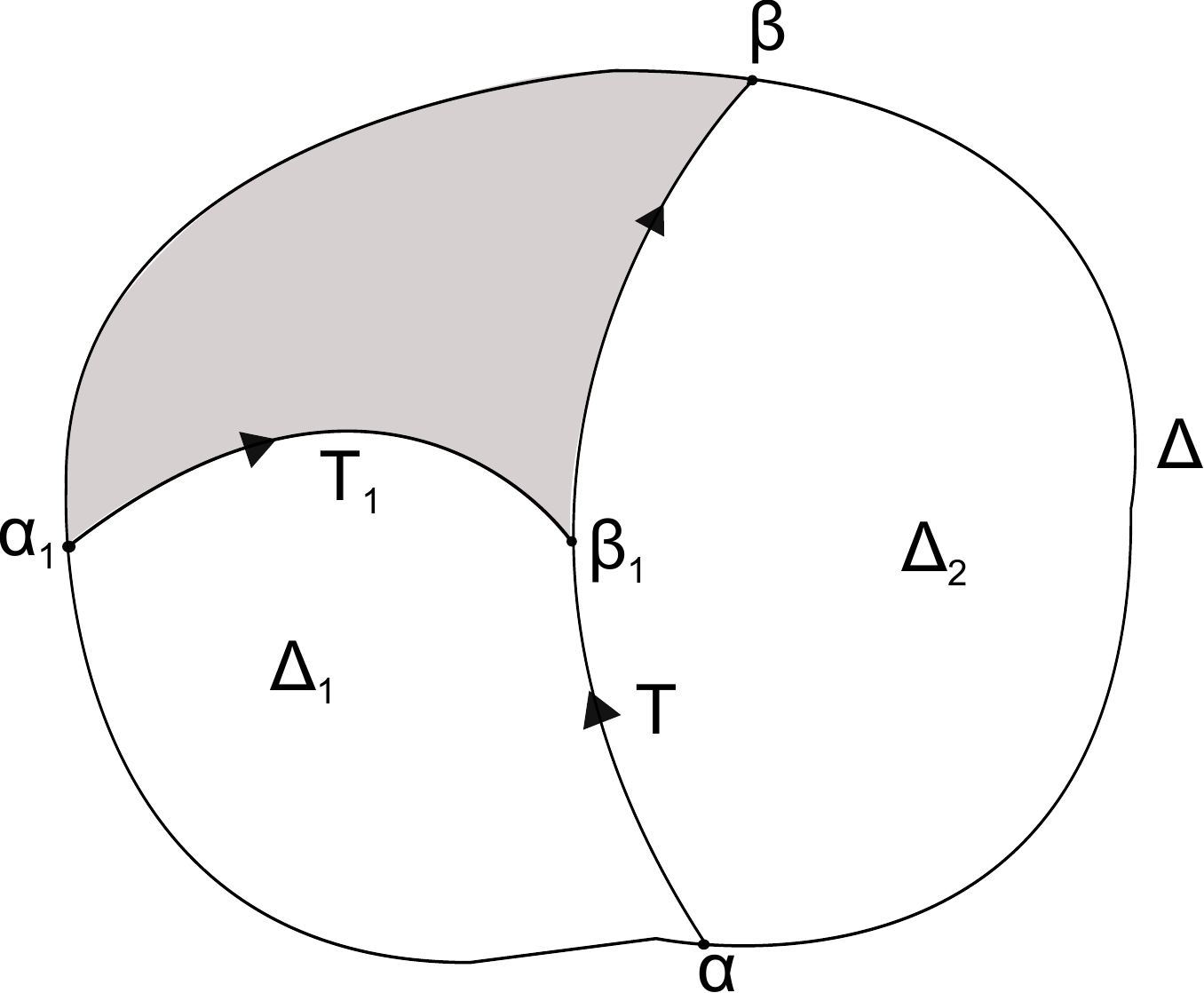}
\caption{ }
\label{fig4}
\end{figure}

 \textit{Case 4:} Suppose that ${\alpha}_1$ lies on the path $T$ and
 ${\beta}_1$
 lies on the boundary $\partial\Delta_1\cap\partial\Delta$. In this
 case let $T_2$ be the composition of the path along $T$ from
 $\alpha$ to ${\alpha}_1$ and the path $T_1$ (see Figure \ref{fig5}).
  Then the
 collection of 2-cells bounded by $T_2$ and the path from ${\beta}_1$ to $\alpha$
  along $\partial\Delta_1\cap\partial\Delta$ forms a
 transversal
 subdiagram of $\Delta$ that contains fewer 2-cells than $\Delta_1$.

\begin{figure}[h!]
\centering
\includegraphics[trim = 0mm 0mm 0mm 0mm, clip,width=2in]{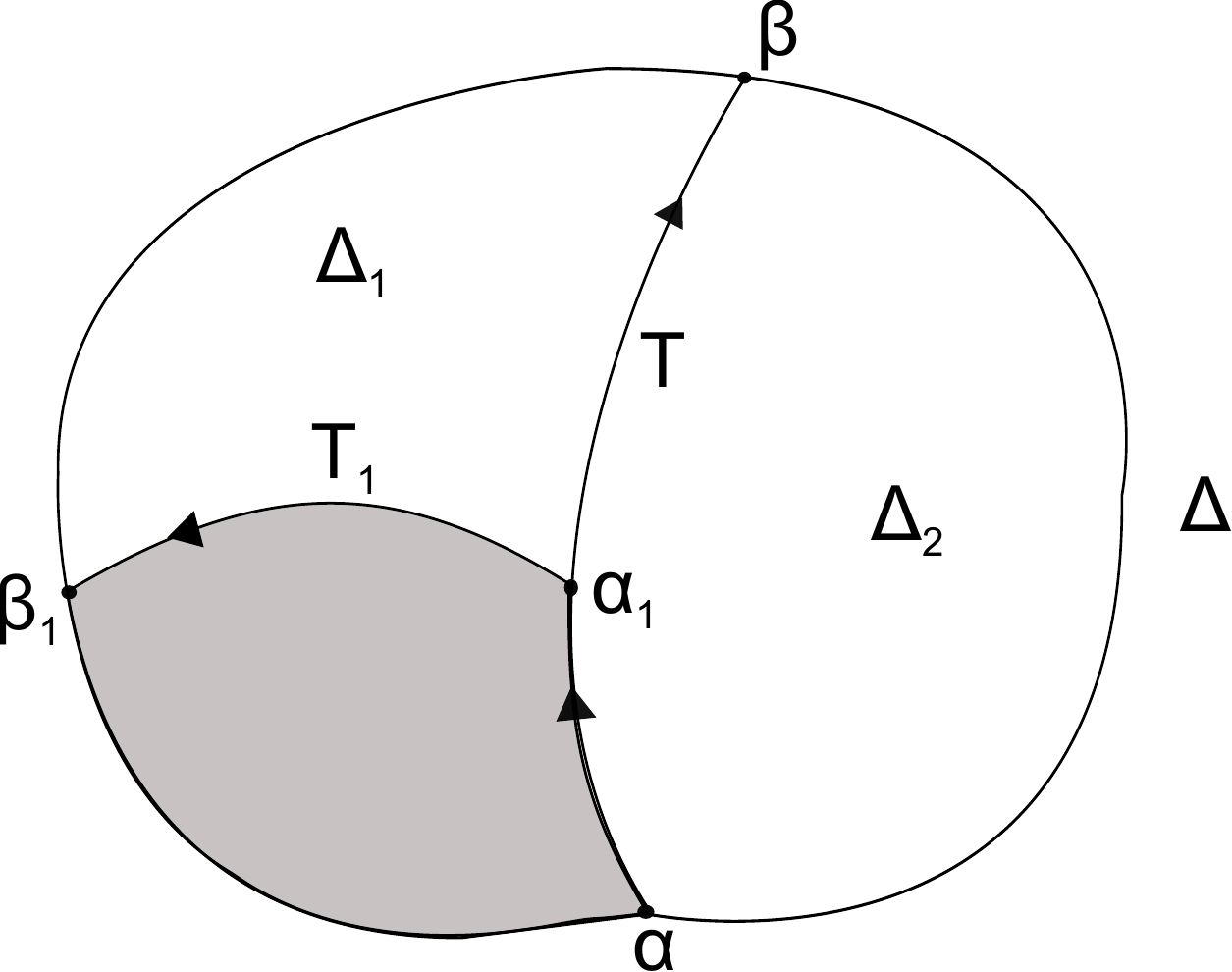}
\caption{ }
\label{fig5}
\end{figure}

Since our original diagram has only finitely many 2-cells, this
process eventually terminates and we get a minimal transversal
subdiagram $D$ of $\Delta$ inside ${\Delta}_1$ that consists of only
one 2-cell. This 2-cell $D$ is bounded by a cycle $pq$ where $p$ is
a directed transversal of ${\Delta}_1$ and $q$ is a path along the
boundary of $\Delta$. Since $D$ is a 2-cell, it has a boundary label
of the form $uv^{-1}$ for some relation $(u,v) \in R$. Since $p$ is
positively labeled, neither the initial vertex nor the terminal
vertex of the 2-cell $D$ can lie on $p$ (or else the orientation of
$p$ would change at that vertex). Hence  both the initial and
terminal vertices of $D$ lie on the boundary of $\Delta$, and so $D$
is a special 2-cell of $\Delta$. A similar argument yields a special
2-cell of $\Delta$ inside ${\Delta}_2$.

\end{proof}

\begin{remark}\label{rem1}

We remark for future use that by Lemma \ref{lem1}, the directed
transversal $p$ constructed in the proof of Lemma \ref{3} above
divides $\Delta$ into two subdiagrams, one of which is the special
2-cell $D$ and the other of which has just one simple component.

\end{remark}

%%\begin{lemma}\label{4}

%%For a van Kampen diagram $\Delta$ over an Adian presentation, a
%%minimal transversal subdiagram is a special 2-cell.

%%\end{lemma}

%%\begin{proof}

%%Let $D$ be a minimal transversal subdiagram of $\Delta$ and denote
%%the associated transversal of $\Delta$ by $T$. By the previous
%%lemma, $D$ consists of only one 2-cell.  Therefore the boundary of
%%$D$ is labeled by $u_iv_i^{-1}$ for some $(u_i,v_i)\in R$. Since $T$
%%is positively labeled, neither the initial vertex nor the terminal
%%vertex of the two-sided cell $D$ can lie on $T$ (or else the
%%orientation of $T$ would change at that vertex). Hence  both the
%%initial and terminal vertices of $D$ lie on the boundary of
%%$\Delta$, and we are done.

%%\end{proof}

 \begin{lemma}\label{5}

If $\Delta$ is a van Kampen diagram over an Adian presentation
$P=(X, R)$ and $\Delta$ has just one simple component
and no extremal vertex, then any word labeling a boundary cycle of
$\Delta$, starting and ending at any vertex 0 on $\partial \Delta$,
is an idempotent in the inverse monoid $S = Inv\langle X \vert R\rangle$.

\end{lemma}

%\newpage

\begin{proof}

We apply  induction on the number of 2-cells of $\Delta$. If
$\Delta$ contains only one 2-cell, then a boundary label of this
2-cell is a cyclic conjugate of $uv^{-1}$, for some $(u,v)\in R$. We
know that $uv^{-1}$ is an idempotent in $S$, because in the inverse
semigroup $S$, $uv^{-1}=vv^{-1}$.
 We  show that any cyclic conjugate of $uv^{-1}$ is also an idempotent in $S$.
Suppose that $u$ factors as a product of two subwords $u \equiv xy$
in $X^+$. We consider the word $y v^{-1}x$ and show that it is an
idempotent in $S$. Note that $y v^{-1}x =y u^{-1}x =y(xy)^{-1}x
=yy^{-1}x^{-1}x$, a product of two idempotents in $S$, hence it is
an idempotent in $S$. We remark that it follows that the inverse of
this word, namely $x^{-1}vy^{-1}$, is also an idempotent of $S$ (in
fact it equals $y v^{-1}x$ in $S$), so all boundary labels of the
2-cell $\Delta$ are idempotents in $S$, no matter which orientation
of the boundary of  $\Delta$ is chosen.

Suppose the conclusion of the lemma is true for all van Kampen
diagrams consisting of one simple component and $k$
 2-cells. Let $\Delta$ be a van Kampen diagram consisting of one simple component and $(k+1)$
2-cells. Then by Lemma \ref{3}, there are at least two special
2-cells in $\Delta$. These special 2-cells do not share an edge that lies on
 $\partial \Delta$. So at least one of them (say the 2-cell $\Pi$),
 has the property that the distinguished vertex 0 is an initial or terminal vertex
  of at least one edge
 in ${\partial} {\Delta} - {\partial} {\Pi}$ (i.e. it  does not lie in the
 ``interior" of $\partial \Pi \cap \partial \Delta$).

 The cell $\Pi$ has
a boundary label $uv^{-1}$ for some $(u,v)\in R$, such that either
$u$ or $v$ labels a path on $\partial \Delta$.  Without loss of
generality, we assume that $u$ labels a path on $\partial \Delta$.
Denote the initial vertex of this path by $\alpha$ and the terminal
vertex by $\beta$. By Remark \ref{rem1}, we may suppose that $v =
srt$ where $s$ labels a path on the boundary of $\Delta$ from
$\alpha$ to some vertex ${\alpha}_1$, $r$ labels a directed
transversal from ${\alpha}_1$ to ${\beta}_1$, and $t$ labels a path
on the boundary of $\Delta$ from ${\beta}_1$ to $\beta$ (see Figure
\ref{fig6}). The distinguished vertex 0 of the van Kampen diagram
$\Delta$ does not lie in the interior of the part of the boundary of
$\Delta$ that is labeled by $tu^{-1}s$. Hence the label on the
boundary cycle of $\partial{\Delta}$ that starts and ends at 0 is of
the form $\ell(\partial \Delta) \equiv gs^{-1}ut^{-1}h$ for some
words $g,h \in (X \cup X^{-1})^*$. Then in the inverse semigroup
$S$, $gs^{-1}ut^{-1}h = gs^{-1}vt^{-1}h = gs^{-1}srtt^{-1}h \leq
grh$, by Proposition \ref{prop0}. But by Lemma \ref{lem1} and Remark
\ref{rem1},  $grh$ labels a boundary cycle of a van Kampen diagram
with only one simple component and fewer 2-cells than $\Delta$.
Hence this word  is an idempotent in $S$ by the induction
hypothesis. Hence $\ell(\partial \Delta) \equiv gs^{-1}ut^{-1}h$ is
an idempotent of $S$, again by Proposition \ref{prop0}. As we
remarked before, if we choose the opposite orientation on the
boundary of $\Delta$, the corresponding boundary label is also an
idempotent of $S$.

\begin{figure}[h!]
\centering
\includegraphics[trim = 0mm 0mm 0mm 0mm, clip,width=2.5in]{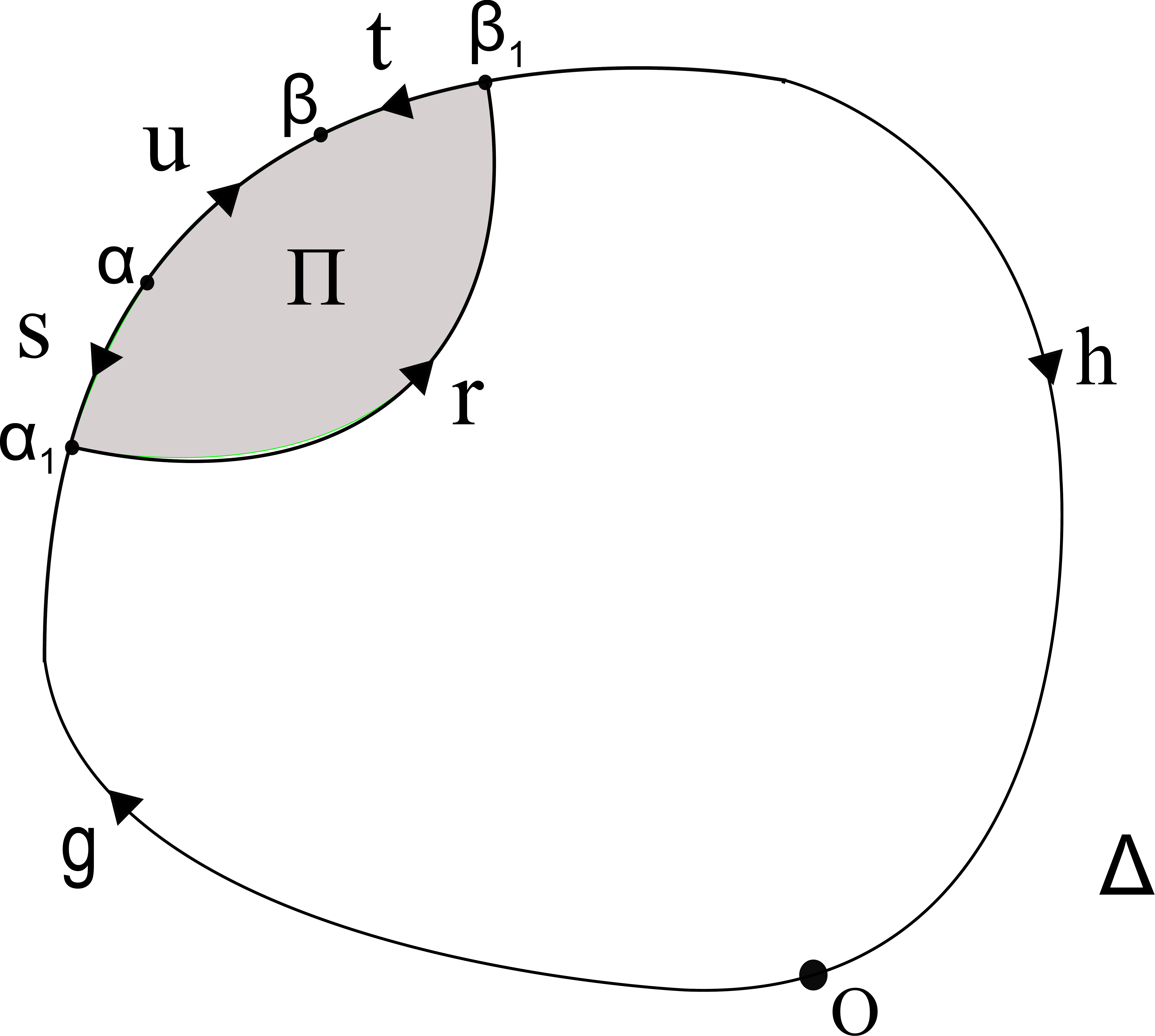}
\caption{ }
\label{fig6}
\end{figure}

%%\textbf{FIGURE}

\end{proof}

\begin{Cor}\label{cor1}
If $\Delta$ is a  van Kampen diagram over an Adian presentation
$P=(X, R)$ that has just one simple component, then any
word labeling a boundary cycle of $D$, starting and ending at any
vertex 0 on $\partial D$, is an idempotent in the inverse monoid $S
= Inv\langle X \vert R\rangle$.
\end{Cor}

\begin{proof}

The van Kampen diagram $\Delta$ either has no extremal vertex, in
which case the result follows from Lemma \ref{5} or else is obtained
from a van Kampen diagram $\Delta'$ with no extremal vertex and just
one simple component by adjoining finitely many finite trees  to the
boundary of  $\Delta'$. The result follows from Proposition
\ref{prop0} and the fact that the label on a boundary cycle of
$\Delta$ is obtained from the label $w$ on a boundary cycle of
$\Delta'$ either by inserting Dyck words (words with reduced form
$1$, which are idempotents in $S$), or if necessary conjugating the
resulting word by some word $u$ if the distinguished vertex 0 is on
an attached tree.

\end{proof}

%\begin{proof}

 \noindent {\bf Proof of Theorem \ref{main}.}  Let $(X, R)$ be an Adian
 presentation and let $w$ be a word such that
  $w=1$ in $G = Gp\langle X \vert R\rangle$. Then $w$ labels a boundary cycle of a
  van Kampen diagram $\Delta$ starting and ending at some designated
  vertex $0$ on $\partial \Delta$.  We show that $w$ is an idempotent of
   $S = Inv \langle X \vert R \rangle$ by
  induction on the number of simple components of $\Delta$. If
  $\Delta$ has just one simple component, this follows from Corollary
  \ref{cor1}. Of course if $D$ has no simple components, then it is
  a finite tree, whose boundary label is a Dyck word, which is an
  idempotent of $S$.

Suppose that $\Delta$ has  $k > 1$ simple  components.
  Assume inductively that the word labeling a boundary cycle of any  van Kampen
  diagram with fewer than $k$ simple components, starting and ending at any
  vertex on its boundary, is an idempotent in $S$.

   It follows from the fact that $\Delta$ is simply connected that $\Delta$ has a cut
   vertex $\gamma$, i.e.,
   a vertex  on the boundary of $\Delta$ whose deletion
  separates  $\Delta$ into two or more connected components $K_1,
  K_2, \ldots K_n$. For each $i$, let $D_i$ be the 2-complex $D_i =
  K_i \cup \{{\gamma}\}$. Then each $D_{i}$ is a
   van Kampen diagram over $(X, R)$. The cut vertex can be chosen so
   that each subdiagram $D_i$ has fewer simple components than $\Delta$ (see Figure \ref{fig7}).

  Then $\Delta = D_1 \vee D_2 \vee \ldots \vee
  D_n$ is obtained by joining the diagrams $D_i$ at the vertex
  $\gamma$.

\begin{figure}[h!]
\centering
\includegraphics[trim = 0mm 0mm 0mm 0mm, clip,width=3.5in]{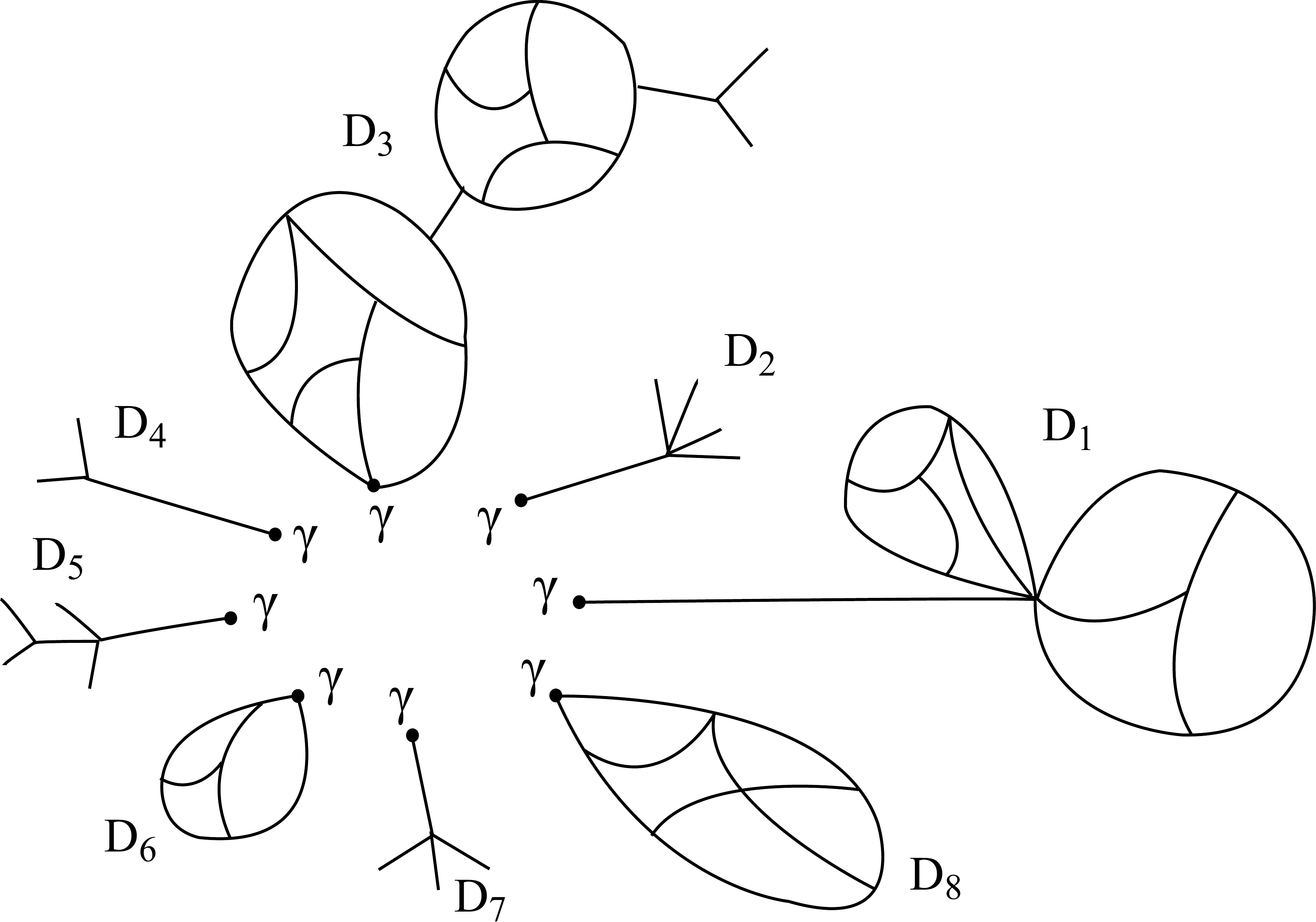}
\caption{ }
\label{fig7}
\end{figure}

 Let $s_i = \ell (q_i)$ be the label on the boundary cycle $q_i$
  around $D_i$ that
  starts and ends at $\gamma$ and respects the orientation of the boundary cycle $p$ of
  $\Delta$ labeled by $w$.
  Then $s_i$ is an idempotent of $S$ by the
  induction hypothesis.

  If 0 = $\gamma$, then the  boundary cycle $p$  is labeled by some permutation of the word
  $s_1 s_2 \ldots s_n$, a
  product of (commuting) idempotents in $S$. Therefore in $S$, $w = s_1 s_2 \ldots s_n$, and hence
  $w$ is an idempotent in $S$.

  %%by Remark \ref{rem2}.

  If 0 is not equal to $\gamma$,
  then it must be a vertex on the
  boundary of exactly one of the  van Kampen diagrams  $D_i $.
  Let $s_i \equiv x_i y_i$ where $x_i$
  labels the portion of the boundary cycle $q_i$ of $D_i $ starting at
  $\gamma$ and ending at 0, and $y_i$ labels the portion of  $q_i$
  starting at 0 and ending at
  $\gamma$ in the orientation of $q_i$.
  Then the  boundary cycle $p$ of $\Delta$ labeled by $w$  is equal in $S$ to
  $ y_i ({\Pi}_{j \neq i}s_j) x_i$, where the product ${\Pi}_{j \neq i} s_j$
  is an idempotent of $S$. Hence
  $w = y_i({\Pi}_{j\neq i}s_j) x_j \leq y_i x_i$ in $S$, by Proposition \ref{prop0}. But by the
  induction assumption, $y_i x_i$ is an idempotent of $S$ since it is
  the label of a boundary cycle
   of $D_i $ starting at 0. It follows
  that the word $w$  is an
  idempotent in $S$, again by Proposition \ref{prop0}.

  Hence
  $S$ is $E$-unitary since the minimum group
  congruence $\sigma$ on $S$ is idempotent-pure.

\hfill $\Box$

\bigskip

We close the paper by showing that inverse monoids naturally
associated with Adian inverse semigroups of the type considered
above are also $E$-unitary. We first record a general fact about
positive presentations of inverse semigroups.

%%Suppose $(X,R)$ is an Adian presentation, and
%%$U=\{(u_iv_i^{-1},1)|(u_i,v_i)\in R\}$. Denote the inverse monoid
%%given by the presentation $Inv\langle X : U\rangle$ by $M_1$.

\begin{Prop}\label{prop1}

Let $S = Inv\langle X \vert R \rangle$ be a positive presentation of
an inverse semigroup, where $R=\{(u_i,v_i) \vert i\in I\}$ and
 $u_i,v_i\in X^+$. Then if $S$ is $E$-unitary, so is the inverse
 monoid $M = Inv\langle X \vert u_{i}v_{i}^{-1} = 1, \, i \in I \rangle$.
\end{Prop}

\begin{proof}
Clearly both inverse semigroups $S$ and $M$ have the same maximal
group homomorphic image $G = Gp \langle X \vert R \rangle$.

If $u_iv_i^{-1}=1$ in $M$, then $v_i=u_i(v_i^{-1}v_i)\leq u_i$ in $M$. Also
$v_iu_i^{-1}=1$, which implies by the same argument that that
$u_i\leq v_i$. Hence $u_i=v_i$ in $M$. It follows from a standard
argument that the natural map from $S$ onto $G$ factors through $M$.
Hence if  $\sigma$ and $\sigma_1$ denote the natural maps from $S$
and $M$ onto $G$ respectively, the fact that $\sigma$ is
idempotent-pure implies that $\sigma_1$ is idempotent-pure, and so
$M$ is $E$-unitary.

%%Let $\rho_1$ be the congruence relation on $FIM(X)$, generated by
%%$\{(u_iv_i^{-1},1)|(u_i,v_i)\in R\}$ and $\rho$ be the congruence
%%relation on $FIM(X)$, generated by $\{(u_i,v_i)|(u_i,v_i)\in R\}$.
%%Then from the above, we have  $\rho\subseteq \rho_1$. But $FIM(X)/
%%\rho_1= M_1$ and $FIM(X)/ \rho = M$.  So, by Corollary 1.6(a) of
%%\cite{CP}, there exits a unique homomorphism $\phi :M \to M_1$. Note
%%that $\phi$ is surjective. We know that both inverse monoids $M$ and
%%$M_1$ have the same maximal group image $G$.

%%Let $\sigma$ and $\sigma_1$ be the maximal group homomorphisms from
%%$M$ and $M_1$ to $G$, respectively. Then the following diagram
%%commutes.

%%To show that $M_1$ is E-unitary, we show that $\sigma_1$ is
%%idempotent pure. For this, we need to show that
%%$\sigma_1^{-1}(1)\subseteq E(M_1)$. Now
%%$\phi^{-1}(\sigma_1^{-1}(1))$ is a subset of $M$. By the
%%commutativity of the above diagram, it follows that
%%$\sigma(\phi^{-1}(\sigma_1^{-1}(1)))=1$ in $G$. But $\sigma$ is
%%idempotent pure. This implies that
%%$\phi^{-1}(\sigma_1^{-1}(1))\subseteq E(M)$, and hence
%% $\sigma_1^{-1}(1)\subseteq E(M_1)$.

\end{proof}

\begin{Cor}
 If $\langle X \vert R \rangle$ is an Adian presentation where $R=\{(u_i,v_i) \vert i\in
 I\}$, then $M = Inv\langle X \vert u_{i}v_{i}^{-1} = 1, \, i \in I
 \rangle$ is $E$-unitary.

\end{Cor}

\begin{proof}
This is immediate from Theorem \ref{main} and  Proposition
\ref{prop1}.
\end{proof}

\end{document}